\pdfoutput=1
\documentclass{amsart}
\usepackage[ascii]{inputenc}
\usepackage{amssymb,comment}
\usepackage[dvipsnames]{xcolor} 
\usepackage{graphicx,enumerate} 
\usepackage[all]{xy} 
\usepackage{mathtools} 
\mathtoolsset{centercolon} 
\xyoption{rotate}
\xyoption{pdf}
\newtheorem{theorem}[equation]{Theorem}
\newtheorem{lemma}[equation]{Lemma}
\newtheorem{cor}[equation]{Corollary}
\newtheorem{question}{Question}
\theoremstyle{definition}
\newtheorem{definition}[equation]{Definition}
\newtheorem{remark}[equation]{Remark}
\numberwithin{equation}{section}
\usepackage[hidelinks]{hyperref}

\DeclareMathOperator{\cf}{cof}

\DeclareMathOperator{\comp}{comp}
\newcommand{\COB}{\mathsf{COB}}
\newcommand{\LCU}{\mathsf{LCU}}
\newcommand{\DOM}{\mathsf{DOM}}
\DeclareMathOperator{\add}{add}
\DeclareMathOperator{\cov}{cov}
\DeclareMathOperator{\non}{non}
\DeclareMathOperator{\cof}{cof}

\newcommand{\Null}{\mathcal N}
\newcommand{\Meager}{\mathcal M}
\newcommand{\addN}{{\ensuremath{\add(\Null)}}}
\newcommand{\cofN}{{\ensuremath{\cof(\Null)}}}
\newcommand{\covN}{{\ensuremath{\cov(\Null)}}}
\newcommand{\nonN}{{\ensuremath{\non(\Null)}}}
\newcommand{\addM}{{\ensuremath{\add(\Meager)}}}
\newcommand{\cofM}{{\ensuremath{\cof(\Meager)}}}
\newcommand{\covM}{{\ensuremath{\cov(\Meager)}}}
\newcommand{\nonM}{{\ensuremath{\non(\Meager)}}}

\newcommand{\cfrak}{\mathfrak{c}}
\newcommand{\bfrak}{\mathfrak{b}}
\newcommand{\dfrak}{\mathfrak{d}}

\newcommand{\hfrak}{\mathfrak{h}}

\newcommand{\mfrak}{\mathfrak{m}}
\newcommand{\pfrak}{\mathfrak{p}}

\newcommand{\innitialmark}[1]{{#1}^\text{pre}}
\newcommand{\finalmark}[1]{{#1}^\text{fin}}
\newcommand{\Ppre}{\innitialmark{P}}
\newcommand{\Pfin}{\finalmark{P}}
\newcommand{\addNi}{\innitialmark{\addN}}
\newcommand{\covNi}{\innitialmark{\covN}}

\newcommand{\nonMi}{\innitialmark{\nonM}}

\newcommand{\bfraki}{\innitialmark{\bfrak}}

\newcommand{\cfraki}{\innitialmark{\cfrak}}
\newcommand{\addNf}{\finalmark{\addN}}
\newcommand{\covNf}{\finalmark{\covN}}
\newcommand{\nonNf}{\finalmark{\nonN}}
\newcommand{\cofNf}{\finalmark{\cofN}}

\newcommand{\covMf}{\finalmark{\covM}}
\newcommand{\nonMf}{\finalmark{\nonM}}

\newcommand{\bfrakf}{\finalmark{\bfrak}}
\newcommand{\dfrakf}{\finalmark{\dfrak}}
\newcommand{\cfrakf}{\finalmark{\cfrak}}

\AtBeginDocument{%
   \def\MR#1{}
}

\title{Cicho\'n's maximum without large cardinals}

\author{Martin Goldstern}
\address{Institut f\"ur Diskrete Mathematik und Geometrie, TU Wien, 1040 Vienna, Austria.}
\email{martin.goldstern@tuwien.ac.at}
\urladdr{http://www.tuwien.ac.at/goldstern/}

\author{Jakob Kellner}
\address{Institut f\"ur Diskrete Mathematik und Geometrie, TU Wien, 1040 Vienna, Austria.}
\email{kellner@fsmat.at}
\urladdr{http://dmg.tuwien.ac.at/kellner/}

\author{Diego A. Mej\'{i}a}
\address{Creative Science Course (Mathematics), Faculty of Science, Shizuoka University, Ohya 836, Suruga-ku, Shizuoka-shi, Japan 422-8529.}
\email{diego.mejia@shizuoka.ac.jp}
\urladdr{http://www.researchgate.com/profile/Diego\_Mejia2}

\author{Saharon Shelah}
\address{Einstein Institute of Mathematics, The Hebrew University of Jerusalem, Jerusalem, 91904, Israel, and Department of Mathematics, Rutgers University, New Brunswick, NJ 08854, USA.}
\email{shlhetal@math.huji.ac.il}
\urladdr{http://shelah.logic.at}

\thanks{This work was supported by the following grants:
Austrian Science Fund (FWF): project number I3081, P29575
(first author) and P26737, P30666
(second author);
Grant-in-Aid for Early Career Scientists 18K13448, Japan Society for the Promotion of Science (third author); European Research Council grant 338821 (fourth author). This is publication number 1177 of the fourth author.}

\date{April 16, 2020}
\begin{document}
\begin{abstract}
    Cicho\'n's diagram lists twelve
    cardinal characteristics (and the provable inequalities between them) associated with the 
    ideals of null sets,
    meager sets, countable sets, and $\sigma$-compact subsets of the irrationals.
    
    It is consistent that all 
    entries of Cicho\'n's diagram are pairwise different
    (apart from $\addM$ and $\cofM$, which are provably 
    equal to other entries).
    However, the consistency proofs 
    so far required large cardinal assumptions.

    In this work, we show the consistency without
    such assumptions.
\end{abstract}

\maketitle

\section*{Introduction}
How many Lebesgue null sets do we need to
cover the real line?
Countably many
are not enough,
as the countable union of null sets is null; 
and continuum many are enough,
as $\bigcup_{r\in\mathbb{R}}\{r\}=\mathbb{R}$.

The answer to this question (and similar ones) is called a
\emph{cardinal characteristic} (sometimes also called cardinal invariant); 
in our case the
characteristic is called ``$\covN$''.

As we have argued, $\aleph_0<\covN\le 2^{\aleph_0}$.
So
if the Continuum Hypothesis (CH) holds, then
$\covN=2^{\aleph_0}$.
It has been shown by G\"odel~\cite{goedel} and Cohen~\cite{MR0157890} that
CH is independent of ZFC. I.e., one can prove: If ZFC is consistent,
then so is ZFC$+$CH as well as ZFC$+\lnot$CH.

Under $\lnot$CH, $\covN$ could be some cardinal
less than $2^{\aleph_0}$, and one can indeed show that
$\aleph_1=\covN=2^{\aleph_0}$,
$\aleph_1<\covN=2^{\aleph_0}$ and
$\aleph_1=\covN<2^{\aleph_0}$ are all consistent.

Some more characteristics associated with the $\sigma$-ideal $\mathcal N$ of null sets are defined:
\begin{itemize}
\item
$\addN$ is the smallest number of null sets whose
union is not null.
\item $\nonN$ is the smallest cardinality
of a non-null set.
\item $\cofN$
is the smallest size of a cofinal family of null sets,
i.e., a family that contains for each null set $N$
a superset of $N$.
\end{itemize}
Replacing $\mathcal N$
with another $\sigma$-ideal $I$ gives us the analogously
defined characteristics for $I$.
In particular, for the meager ideal $\mathcal M$ we get
$\addM$, $\nonM$, $\covM$, $\cofM$.

For the $\sigma$-ideal $\texttt{ctbl}$ of countable sets, it is easy to see that 
$\add(\texttt{ctbl})=\non(\texttt{ctbl})=\aleph_1$ and $\cov(\texttt{ctbl})=\cof(\texttt{ctbl})=2^{\aleph_0}$, which is 
also called $\cfrak$ (for ``continuum'').

For $\mathcal K$, the $\sigma$-ideal generated by the
compact subsets of the irrationals, it turns out that
$\add(\mathcal K)=\non(\mathcal K)$. This characteristic is more commonly called
$\bfrak$. We also have 
$\cov(\mathcal K)=\cof(\mathcal K)$, called $\dfrak$.

These characteristics are customarily displayed
in Cicho\'n's  diagram, see Figure~\ref{fig:cichon}.
\newcommand{\mye}{*+[F.]{\phantom{\lambda}}}
\begin{figure}
\resizebox{\textwidth}{!}{$
\xymatrix@=2ex{
&
&
&\cfrak
\\
\covN\ar[r]
&\nonM\ar[r]
&\cofM\ar[r]
&\cofN\ar[u]
\\
&
\mathfrak b\ar[r]\ar[u]
&\mathfrak d\ar[u]
\\
\addN\ar[r]\ar[uu]
&\addM\ar[r]\ar[u]
&\covM\ar[r]\ar[u]
&\nonN\ar[uu]
\\
\aleph_1\ar[u] 
}\quad
\xymatrix@=2ex{
&
&
&\cfrak
\\
\covN\ar[r]
&\nonM\ar[r]
&\mye\ar[r]
&\cofN\ar[u]
\\
&
\mathfrak b\ar[r]\ar[u]
&\mathfrak d\ar[u]
\\
\addN\ar[r]\ar[uu]
&\mye\ar[r]\ar[u]
&\covM\ar[r]\ar[u]
&\nonN\ar[uu]
\\
\aleph_1\ar[u] 
}
$}
    \caption{\label{fig:cichon}Cicho\'n's diagram (left). In the version on the right, the
two ``dependent'' values are removed; the ``independent'' ones remain (nine entries excluding $\aleph_1$, or ten including it).
It is consistent that these ten entries are pairwise different.}
\end{figure}
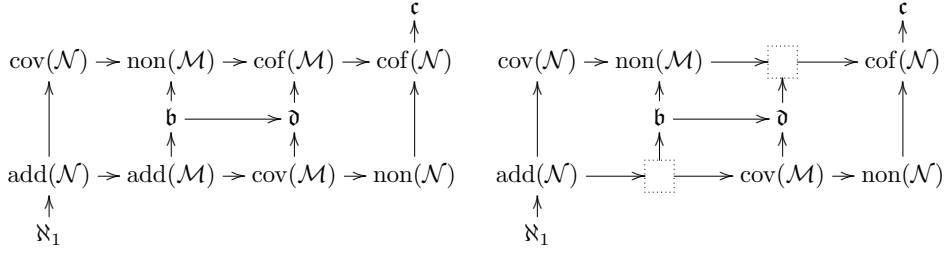
An arrow from $\mathfrak x$ to $\mathfrak y$ indicates that ZFC proves
$\mathfrak x\le \mathfrak y$.
Moreover, one can show that $\max\{\mathfrak d,\nonM\}=\cofM$ and $\min\{\mathfrak b,\covM\}=\addM$.
A series of results~\cite{MR719666,MR781072,MR1233917,MR1071305,MR1022984,MR613787,MR735576,MR697963,MR800191},
summarized in~\cite[Ch.~7]{BJ},
proves these (in)equalities in ZFC and shows that they are the only ones
provable. More precisely, all assignments of the values $\aleph_1$ and $\aleph_2$
to the nine ``independent'' characteristics in Cicho\'n's diagram (excluding $\aleph_1$ and
including $\cfrak$) 
are consistent with ZFC, provided they honor the inequalities given by the arrows.

This leaves the question on how to separate more than two entries simultaneously. There was a lot of progress
in recent years, giving four and up to seven values~\cite{MR3047455,five,MR3513558,MR3796283,mejiavert}.
Finally, it was shown~\cite{GKS} that the following statement, which we call \textbf{``Cicho\'n's maximum''}, is consistent:
\begin{quote}
The maximal possible number of  entries of Cicho\'n's diagram, i.e., all ten ``independent'' entries (including $\aleph_1$ and $\cfrak$), are pairwise different.
\end{quote}
However, the proof required four Boolean ultrapower embeddings, constructed from four strongly compact cardinals.\footnote{
A simpler example of this Boolean ultrapower construction,  giving only
eight different values and using three compacts, can be found in~\cite{KTT}; and later a
construction for Cicho\'n's maximum requiring only three
compacts was given 
in~\cite{diegoetal}.
\cite{moti2} notes that superstrongs are sufficient
for the constructions. However until now all proofs 
showing the consistency of eight or more
different values needed some 
large cardinals assumptions.}
A strongly compact cardinal is an example of a so-called ``large cardinal'' (LC). Such cardinals turned out to be an important 
scale for measuring consistency strengths of mathematical (and in particular set theoretic) statements:
There are 
many examples of statements $\varphi$
where one cannot prove
\begin{quote}
The consistency of ZFC implies the consistency of (ZFC plus $\varphi$),
\end{quote}
but only:
\begin{quote}
The consistency of (ZFC plus LC) implies the consistency of (ZFC plus $\varphi$)
\end{quote}
for some specific large cardinal axiom LC.
In many cases, one can even show that $\varphi$ is equiconsistent to LC (i.e., one can also prove that the consistency of (ZFC plus $\varphi$) implies the consistency of (ZFC plus LC)). For example, ``there is an extension of Lebesgue measure to a $\sigma$-complete measure which measures \emph{all} sets of reals'' is equiconsistent with a so-called measurable cardinal (a notion much weaker than a strongly compact). 

In case of $\varphi$ being Cicho\'n's maximum, we
previously could only prove an upper bound for 
the consistency strength, but conjectured that Cicho\'n's maximum is
actually  equiconsistent with ZFC. This turns out to be correct.


In this work, we introduce a new method to control cardinal characteristics when modifying a finite support ccc iteration (by taking intersections with $\sigma$-closed elementary submodels). 
This method can replace the Boolean ultrapower embeddings in previous constructions, 
so in particular we can get Cicho\'n's maximum without assuming large cardinals. Furthermore, we can get arbitrary regular cardinals as the values of the entries 
in Cicho\'n's diagram. As the method is quite general, we expect that it can be applied to
control the values of other characteristics, in other constructions, as well.  

This paper should be reasonably self-contained
(modulo an understanding of forcing, such as
presented in~\cite{Kunen}).
However,
in Section~\ref{sec:prep}  we just 
quote the result (from \cite{GKS} or alternatively from \cite{diegoetal})
that a suitable preparatory forcing $\Ppre$
for the left hand side exists, without proofs or much explanation.

\bigskip

Annotated contents:
\begin{itemize}
    \item[S.~\ref{sec:cob}]
    We define the properties $\LCU$ and $\COB$  for a forcing
    $P$, which give us
    the ``strong witnesses'' that will guarantee
    the desired equalities (or rather: both sides of the required inequalities) for the respective cardinal characteristics.
    We show how these properties are preserved when intersecting $P$ with a $\sigma$-complete elementary submodel.
    \item[S.~\ref{sec:prep}]
    We just quote (without proof) the result 
    from~\cite{GKS} (or~\cite{diegoetal})
    that a suitable forcing $\Ppre$ for the left hand side 
    with suitable $\LCU$ and $\COB$ properties exists.
    \item[S.~\ref{sec:new}] 
    We prove the main result: There is a complete subforcing $\Pfin$ of $\Ppre$ which forces ten different values to Cicho\'n's diagram (we can actually choose
    any desired regular values).
    \item[S.~\ref{sec:more}] We remark that the same argument can
    be applied to alternative ``initial forcings'' for the left hand side.
    In particular, using a construction of~\cite{KeShTa:1131}, we get  
    another ordering of the ten entries in Cicho\'n's diagram.
    \item[S.~\ref{sec:q}] We list some open questions regarding alternative orders of Cicho\'n's diagram with ten values.
\end{itemize}

\section{The \texorpdfstring{$\LCU$}{LCU} and \texorpdfstring{$\COB$}{COB} properties and  \texorpdfstring{$\sigma$}{sigma}-closed
elementary submodels}\label{sec:cob}

Let $R$ be a binary relation on some basic set $Y$.
The cardinal
$\bfrak_R$, the \emph{bounding number}  of $R$, is the minimal size of an unbounded family. I.e.,
\[
\bfrak_R:= \min\bigl\{|B|:\, B \subseteq Y,\ ( \forall g\in Y)\, (\exists f\in B)\, \lnot (f R g)\bigr\},
\]
Dually, $\dfrak_R$, the  \emph{dominating number} of $R$, is the minimal size of a dominating family. I.e.,
\[\dfrak_R:= \min\bigl\{|D|:\, D \subseteq Y,\ ( \forall f\in Y)\, (\exists g\in D)\, f R g\bigr\}.\]

We will use these notions in two situations:

On the one hand, $R$ may be a directed partial order (or a linear order) without largest element, such as $([X]^{<\kappa}, \subseteq)$ or $(\kappa,\in)$.
Then we will call $\bfrak_R$ the \emph{completeness of $R$} and denote it by $\comp(R)$;
and we call $\dfrak_R$ the \emph{cofinality of $R$} and denote it by $\cf(R)$. 
Note that $R$ is
${<}\lambda$-directed iff $\lambda\le\comp(R)$
(as we assume that $R$ is directed).
If in addition $R$ is linear without a maximal element, then $\cf(R)=\comp(R)$ is an infinite regular cardinal. 
%

On the other hand, $R$ may be a  (possibly non-transitive) Borel relation on the reals 
(more generally: a sufficiently absolute definition of a binary relation 
on the reals), and we get the cardinal characteristics of the continuum
$\bfrak_R$ and $\dfrak_R$.
Note that $(\bfrak_R, \dfrak_R)=(\dfrak_{R^\perp},
\bfrak_{R^\perp})$, where we define the dual relation $R^\perp$
by $x R^\perp y$
iff $\lnot(y R x)$.
All entries of Cicho\'n's diagram are of this form, for quite natural 
relations~$R$. (For more details, see the references after Theorem~\ref{thm:left}.)

In the following we give definitions of $\LCU$ and
$\COB$ which are notational variants\footnote{There are other variants of these definitions
that do not mention forcings (\cite[Def.~2.11]{1166}) but are applied to the extension $V[G]$. These variants are basically equivalent.}
of the definitions given in~\cite[Def.~1.8 \&~1.15]{GKS}.

We investigate relations on the reals, and fix  $\omega^\omega$
as representation of the reals.
(This choice is irrelevant, and we could use any of the other usual representations as well. We just pick one
so that we can later refer to the reals as a well defined object,
and so that we can e.g. use $(\forall x\in\omega^\omega)$ in formulas.)

\begin{definition}\label{def:cob}
Assume $R$ is a binary relation on $\omega^\omega$ which
is Borel, or just sufficiently absolutely defined.\footnote{The discussion after~\eqref{eq:absolute} shows which amount of absoluteness is sufficient for us. We will need non-Borel relations only in Subsection~\ref{ss:other}.}
\begin{itemize}
    \item For a directed partial order $(S, \le_S) $ without maximal elements, the ``cone of bounds'' property $\COB(P,S)$ says:
    There is a sequence\footnote{If $S$ is a (partially) ordered set, we sometimes use ``a sequence indexed by $S$'' as synonym for ``a function with domain $S$''.} $(g_s)_{s\in S}$ of $P$-names of reals such that
    for any $P$-name $f$ of a real there is an $s\in S$ such that \[P\Vdash (\forall t\ge_S s)\, f R g_t.\]

    \item For a linear order $L$ without largest element, the ``linear cofinal unbounded''
    property $\LCU_R(P,L)$ is defined as:

    There is a sequence $(c_\alpha)_{\alpha\in L}$ of $P$-names  of reals
    such that for each $P$-name $g$ of a real there is an $\alpha_0\in L$ such that
    \[ P\Vdash (\forall \alpha\ge_L \alpha_0)\,\lnot (c_\alpha R g).\]
\end{itemize}
\end{definition}
(When writing $P\Vdash fRg$, we of course mean
that we evaluate the definition of $R$ in
the extension.)

Actually, $\LCU$ is a special case of $\COB$:
\begin{equation}\label{eq:LCUisCOB}
\LCU_R(P, \kappa)\text{ is equivalent to }\COB_{R^\perp}(P, \kappa)
\end{equation}
(again, $R^\perp$ denotes the dual of $R$).
However, $\LCU$ and $\COB$  will play different roles in our arguments, so we prefer to have different notations for these two concepts.

The following is basically the same as~\cite[Lem.~1.9 \&~1.16]{GKS} (see also~\cite[Fact~2.14]{1166}):
\begin{lemma}\label{lem:blubb}
\begin{enumerate}
\item 
Let $S$ be a ${<}\lambda$-directed partial order without a  largest element, and let  $A\subseteq S$ be cofinal.
Then $\COB_R(P,S)$ is equivalent to $\COB_R(P,A)$, and  implies
\[
  P\Vdash \bigl(\, \bfrak_R\ge\lambda \ \&\ \dfrak_R\le |A|\, \bigr).
\]
\item 
Let $L$ be linear without a  largest element and set  $\lambda:=\cf(L)$. (So $\lambda$ is an infinite regular cardinal.)
Then $\LCU_R(P,L)$ is equivalent to
$\LCU_R(P,\lambda)$, and implies\footnote{We actually do mean $\dfrak_R\ge \lambda$ and not just $\dfrak_R\ge |\lambda|$,
i.e., if $\lambda$ is not a cardinal in the extension anymore,
then we have $\dfrak_R\ge |\lambda|^+$. But this is irrelevant
in our application, as $P$ will preserve $\lambda$.}
\[
  P\Vdash \bigl(\, \bfrak_R\le|\lambda| \ \&\ \dfrak_R\ge \lambda\, \bigr).
\]
\end{enumerate}
\end{lemma}

\begin{proof}
    Regarding the equivalence:
    Let $(g_s)_{s\in S}$ witness $\COB_R(P,S)$.
    Then $(g_s)_{s\in A}$ witnesses $\COB_R(P,A)$.
    On the other hand, if $(g'_s)_{s\in A}$ witnesses $\COB_R(P,A)$, then we assign to every $s\in S$
    some $a(s)\in A$ above $s$, and set
    $g''_s:=g'_{a(s)}$. Then $(g''_s)_{s\in S}$ witnesses $\COB_R(P,S)$.

    From now on assume that $(g_s)_{s\in A}$ witnesses $\COB_R(P,A)$.
    Regarding $\dfrak_R$, note that $\{g_s:\, s\in A\}$ is forced to be dominating.
    
    Regarding $\bfrak_R$, assume that
    $p_0$ forces that $X\subseteq \omega^\omega$
    is of size less than (the ordinal)~$\lambda$.
    Fix $p_1\leq p_0$, $\kappa<\lambda$ and $P$-names
    $(f_\alpha)_{\alpha\in\kappa}$ of reals such that
    $p_1\Vdash X=\{f_\alpha:\, \alpha\in\kappa\}$.
    For each $\alpha$ let $s_\alpha$ be an element of $S$
    satisfying the $\COB$ requirement for $f_\alpha$.
    As $S$ is ${<}\lambda$-directed, there is some $t\in S$
    above all $s_\alpha$, i.e., $P\Vdash f_\alpha R g_t$ for all $\alpha\in \kappa$. Accordingly, $p_0$ cannot force $X$ to be unbounded.

    The claims on $\LCU$ follow from the ones on $\COB$
    by~\eqref{eq:LCUisCOB} (together with the fact that
    for linear orders $L$, $\comp(L)=\cof(L)$ and
    that $(\bfrak_R,\dfrak_R)=(\dfrak_{R^\perp},\bfrak_{R^\perp})$).
\end{proof}

In the following results we show that when we restrict a poset $P$ to a $\sigma$-closed
elementary submodel $N$ of some $H(\chi)$, then 
the $\LCU$ and $\COB$ properties still hold (when we intersect the parameter with $N$ as well).
These are simple technical tools we will use to prove the main results.

Assume that $\kappa$ is regular,
$P$ $\kappa$-cc,
$N\preccurlyeq H(\chi)$ is ${<}\kappa$-closed
and $P\in N$.
Then $P\cap N$ is again $\kappa$-cc
and thus a complete subforcing of $P$.
So given a $P\cap N$-generic $G$ over $V$, there is 
a $P$-generic $G^+$ over $V$ extending $G$.
Note that $G^+$ is $P$-generic over $N$ as well, and that 
$N[G^+]\preccurlyeq H^{V[G^+]}(\chi)$.

There is a correspondence of $P\cap N$-names $\sigma$ for reals and
$P$-names $\tau\in N$ for reals,
such that $\sigma[G]=\tau[G^+]$ and for all $p\in P\cap N$ and sufficiently absolute $\varphi$,
\begin{equation}\label{eq:absolute}
    p\Vdash_P\varphi(\tau)\text{ iff }p\Vdash_{P\cap N} \varphi(\sigma).
\end{equation}
In a bit more detail:
A ``nice $Q$-name for a $\zeta$-subset'' (for an ordinal $\zeta$)
is a sequence 
$\bar{h}:=\big((h_n,A_n)\big)_{n<\zeta}$ such that $A_n$ is a maximal antichain in $Q$
and $h_n:A_n\to 2$ (evaluated in the generic extension as $\{n\in\zeta:\, (\exists a\in G_Q\cap A_n)\, h_n(a)=1\}$).
As $P\cap N\lessdot P$, every nice $P\cap N$-name $\bar h$ for a $\zeta$-subset is also a nice $P$-name, and furthermore $\bar h\in N$ whenever $\zeta<\kappa$ (as $N$ is ${<}\kappa$-closed).
On the other hand, if $\zeta<\kappa$ then every nice $P$-name $\bar h$ for a $\zeta$-subset which is in $N$ is actually a nice $P\cap N$-name.
Note that if $\varphi$ is Borel, then we are done with
showing~\eqref{eq:absolute}. For a more general formula $\varphi$,
note that we have just shown that 
$N[G^+]\cap 2^{<\kappa}=V[G]\cap 2^{<\kappa}$, and
using an absolute bijection between $2^{<\kappa}$ and $H(\kappa)$, 
we get that $N[G^+]\cap H(\kappa)=V[G]\cap H(\kappa)$.
So~\eqref{eq:absolute} holds whenever
$\varphi$ is, e.g., (provably) absolute between the universe and $H(\chi)$ (for $\chi=\kappa$
as well as for $\chi$ sufficiently large), where $\varphi$ may
use  elements of $H(\kappa)$ (or names for such elements) as parameters.

\begin{lemma}\label{lem:trivialalso}
      Assume $P$ is $\kappa$-cc for some uncountable regular $\kappa$
      and $N\preccurlyeq H(\chi)$ is ${<}\kappa$-closed.
      Then $P\cap N$ is
      a $\kappa$-cc complete subforcing of $P$. Assume in the following that
      $P$, $S$, $L$, $\kappa$, $R$ are in $N$.
      \begin{enumerate}
        \item      $\COB_R(P,S)$ implies   $\COB_R(P\cap N,S\cap N)$.

            So if we set
            $\lambda_1:=\comp(S\cap N)$ and
            $\lambda_2:=\cof(S\cap N)$, then
            \\
            $\COB_R(P,S)$ implies
            $P\cap N\Vdash \bfrak_R\ge \lambda_1
            \ \&\ \dfrak_R\le |\lambda_2|$.
          \item\label{item:lcu} $\LCU_R(P,L)$ implies
          $\LCU_R(P\cap N,L\cap N)$.

          So if we set $\lambda:=\cf(L\cap N)$, then
          \\
          $\LCU_R(P,L)$
           implies
          $P\cap N\Vdash \bfrak_R\le |\lambda|
            \ \&\ \dfrak_R\ge \lambda$.
            
      \end{enumerate}
\end{lemma}

\begin{proof}    
      Let $(f_s)_{s\in S}$ witness $\COB_R(P,S)$ in $N$.
      Then $(f_s)_{s\in S\cap N}$ witnesses $\COB_R(P\cap N,S\cap N)$:
      Assume $g\in V$ is a $P\cap N$-name for a real. As above we interpret it as a $P$-name in $N$.
      So $N$ thinks there is some $s\in S$
      such that for all $t\ge_Ss$, $P \Vdash gRf_t$.
      So by absoluteness~\eqref{eq:absolute}, for every $t\ge_S s$ in $N$
      we get $P\cap N\Vdash gRf_t$.
      
      Again, (2) is a special case of (1).
\end{proof}

\begin{lemma}\label{lem:trivial}
Let $\kappa\leq\lambda\leq\theta$ be cardinals with $\kappa$ and $\lambda$ uncountable regular, $S$ a directed set without maximal elements, $\zeta$ a regular cardinal, and let $P$ be a $\kappa$-cc poset. 
Assume $(N_i)_{i<\lambda}$ is an increasing sequence of ${<}\kappa$-closed elementary submodels of $H(\chi)$,
where $\chi$ is a fixed, sufficiently large\footnote{It is enough to
assume $\theta$, $\zeta$, $S$ and $2^P$ are in $H(\chi)$.} regular cardinal.
Assume that $|N_i|=\theta$, that $\theta\cup\{\theta,P,R,S,\zeta\}\subseteq N_i$, and that $N_i\in N_{i+1}$ for any $i<\lambda$. Set $N:=\bigcup_{i<\lambda}N_i$ (which is also a ${<}\kappa$-closed elementary submodel).
\begin{enumerate}
    \item\label{item:eins} 
    $\cf(\zeta\cap N)=\zeta'$, where 
    $\zeta':=\begin{cases}\zeta&\text{if } \zeta\le\theta,\\\lambda&\text{otherwise.}\end{cases}$
    
    In particular $\LCU_R(P,\zeta)$ implies $\LCU_R(P\cap N,\zeta')$.
    \item\label{item:zwei} $\comp(S\cap N)\ge \min(\kappa,\comp(S))$.
    \item\label{item:drei} If $\cf(S)\le \theta$, then $S\cap N$ is cofinal in $S$,
    and in particular $S\cap N$ has the same cofinality and
    completeness as $S$.
    \item\label{item:zero} If $\comp(S)>\theta$, then $\cof(S\cap N)=\lambda$.
    
    In particular $\COB_R(P,S)$ implies $\COB_R(P\cap N,\lambda)$.
\end{enumerate}
\end{lemma}


\begin{proof}
      For (\ref{item:zwei}), the assumptions of Lemma~\ref{lem:trivialalso} are sufficient:
      Assume that $A\subseteq S\cap N$
      has size less than $\min(\kappa,\comp(S))$. As $N$ is ${<}\kappa$-closed, $A\in N$. By absoluteness, $N$ knows that
      the set $A$ (which is smaller than $\comp(S)$ 
      after all) has an upper bound, so there is an upper bound of $A$
      in $S\cap N$.

      (\ref{item:drei}) only requires that $\theta\cup\{\theta\}\subseteq N$ and $|N|=\theta$:
      In $N$, let $A\subseteq S$ be a cofinal subset of size $\cf(S)$. Since
      $\cf(S)\le \theta\subseteq N$, we have $A\subseteq N$,
      so $A\subseteq S\cap N$ is cofinal in $S$.
      And it is clear that any cofinal subset of a partial order has
      the same completeness and cofinality as the order itself.

    For (\ref{item:zero}),  fix $i<\lambda$. Since $|N_i|\leq\theta<\comp(S)$, there is some $\alpha_i\in S$ bounding $N_i\cap S$. In fact, we can find such $\alpha_i$ in $S\cap N_{i+1}$ because $N_i\in N_{i+1}$. Hence, $(\alpha_i)_{i<\lambda}$ is a cofinal increasing sequence of $S\cap N$, so $\cf(S\cap N)=\lambda$.
    The claim on $\COB$ follows from Lemmas~\ref{lem:trivialalso}(1) and~\ref{lem:blubb}(1).

    For (\ref{item:eins}), if $\zeta>\theta$ then, by (\ref{item:zero}) applied to $S=\zeta$, $\cof(\zeta\cap N)=\lambda$; if $\zeta\leq\theta$ then $\zeta\cap N=\zeta$, so $\zeta'=\zeta$.
    The claim on $\LCU$ follows from Lemmas~\ref{lem:trivialalso}(2) and~\ref{lem:blubb}(2).
\end{proof}

\section{The forcing for the left hand side}\label{sec:prep}

We set $(\bfrak_i,\dfrak_i)$ to be the following
pairs of dual characteristics in Cicho\'n's diagram:
\begin{equation}\label{eq:ordereofchars}
(\bfrak_i,\dfrak_i)=
\left\{
\begin{array}{c@{}c@{}c@{}c@{}c@{}c@{}c}
(&\addN&,&\cofN&)\text{\ \ for }&i=1,\\
(&\covN&,&\nonN&)\text{\ \ for }&i=2,\\
(&\bfrak&,&\dfrak&)\text{\ \ for }&i=3,\\
(&\nonM&,&\ \covM&)\text{\ \ for }&i=4.
\end{array}
\right.    
\end{equation}

We will use for each $i$ two Borel relations\footnote{Actually, 
in most cases we will use the same
$R^\LCU_i$ and $R^\COB_i$, which 
is moreover the ``canonical'' choice
for $(\bfrak_i,\dfrak_i)$. See the explanation that follows Theorem~\ref{thm:left}.} on $\omega^\omega$,
$R^\LCU_i$ and $R^\COB_i$, in such a way that ZFC proves
\begin{equation}\label{eq:COBLCU}
\bfrak_{R_i^\COB}\le\bfrak_i\le\bfrak_{R_i^\LCU}
\text{ and }
\dfrak_{R_i^\COB}\ge\dfrak_i\ge\dfrak_{R_i^\LCU}.
\end{equation}

We write $\LCU_i$ instead of $\LCU_{R^\LCU_i}$
and $\COB_i$ instead of $\COB_{R^\COB_i}$.

It is useful to have relations satisfying~\eqref{eq:COBLCU}, because in this way we get:
\begin{cor}\label{cor:alsotrivial}
$\LCU_i(P,\kappa)$ for $\kappa$ regular implies $P\Vdash \bfrak_i\le |\kappa|\ \& \ \dfrak_i\ge\kappa$.

$\COB_i(P,S)$
for $\comp(S)=\kappa_1$ and $\cf(S)=\kappa_2$
implies $P\Vdash \bfrak_i\ge \kappa_1\ \& \ \dfrak_i\le|\kappa_2|$.
\end{cor}

\begin{theorem}\label{thm:left}
Assume 
GCH and fix regular cardinals $\aleph_1<\mu_1<\mu_2<\mu_3<\mu_4<\mu_\infty$
such that each
$\mu_n$ is the successor of a regular cardinal.

We can choose $R^\LCU_i,R^\COB_i$ satisfying~\eqref{eq:COBLCU} and construct
a ccc poset $P$ such that the following holds for $i=1,2,3,4$:
\begin{itemize}
    \item[(a)] If $i<4$ then, for all regular $\kappa$ such that $\mu_i\le \kappa\le\mu_\infty$, $\LCU_i(P,\kappa)$ holds. In the case $i=4$, $\LCU_i(P,\mu_4)$ and $\LCU_i(P,\mu_\infty)$ hold.
    \item[(b)] There is a directed order $S_i$ with $\comp(S_i)=\mu_i$ and $\cof(S_i)=\mu_\infty$ such that
    $\COB_i(P,S_i)$ holds.
\end{itemize}
Accordingly, $P$ forces
\[
\addN=\mu_1<\covN=\mu_2<\bfrak=\mu_3<\nonM=\mu_4<\covM=\mu_\infty=\cfrak.
\]
\end{theorem}
This theorem is proved in~\cite{GKS}; we will not repeat the 
proof here but instead point out where to find the definitions and proofs in the cited papers (the \emph{italic} labels in the following paragraph refer to the cited paper):

\emph{Def.~1.2} defines relations called $R_i$ for $i=1,\dots,4$.
These $R_i$ are, apart from $i=2$, the ``canonical''
relations for $\bfrak_i,\dfrak_i$. 
They play the role of $R^\LCU_i$ and, apart from $i=2$,
also of $R^\COB_i$.
$R^\COB_2$ is implicitly defined in \emph{Def.~1.17} as the canonical relation: $x R^\COB_2 y$ iff $y$ is not in the Borel null set coded by $x$.
\emph{Lem.~1.3} corresponds to~\eqref{eq:COBLCU} in this work,
and \emph{Thm.~1.35} is our Theorem~\ref{thm:left}.

\begin{remark}\label{rem:BCM}
In~\cite[Thm.~5.3]{diegoetal} a different
construction is presented, which gives a stronger conclusion and requires the weaker assumption that $\aleph_1\leq\mu_1<\mu_2<\mu_3<\mu_4<\mu_\infty=\mu_\infty^{<\mu_3}$ are just regular cardinals.
If we use this paper, then  $R^\COB_i=R^\LCU_i=R_i$ for all $i$, see
~\cite[Exm.~2.16]{diegoetal} (where $R_i$ corresponds to item $(5-i)$).
\end{remark}



\section{Cicho\'n's maximum without large cardinals}\label{sec:new}

\medskip

\begin{theorem}\label{main}
Assume GCH and $(\mu_n)_{1\le n\le 9}$ is a weakly increasing sequence
of cardinals with $\mu_n$ regular for $n\leq8$ and $\mu_9^{\aleph_0}=\mu_9$. Then there is a ccc poset $\Pfin$
forcing that
\begin{gather*}
\aleph_1\leq\addN=\mu_1\leq\covN=\mu_2\leq\bfrak=\mu_3\leq\nonM=\mu_4\leq \\
\leq\covM=\mu_5\leq\dfrak=\mu_6\leq\nonN=\mu_7\leq\cofN=\mu_8\leq\cfrak=\mu_9.
\end{gather*}
\end{theorem}

Full GCH is not actually required, see Remark~\ref{remGCH}.

Note that the $\mu_n$ are required to be only weakly increasing, i.e., we can replace each $\le$ in the inequality of characteristics
 by either $<$ or $=$ at will. So we get 
the consistency of $2^9$ many different ``sub-constellations'' in Cicho\'n's diagram. 
Of course several of these 
have been known to be consistent before (even without large cardinals). E.g., 
the sub-constellation where we always choose ``$=$'' is just CH.

\begin{proof}

We fix an increasing sequence of cardinals (see Figure~\ref{fig:setup2})
\begin{equation}\label{eq:lambalambalamba}
\begin{gathered}
\aleph_1\leq\lambda_7\leq\lambda_5\leq\lambda_3\leq\lambda_1\leq\lambda_0\leq\lambda_2\leq\lambda_4\leq\lambda_6\leq \lambda_\infty<\\
<\theta_7<\theta_6<\theta_5<\theta_4<\theta_3<\theta_2<\theta_1<\theta_0<\theta_\infty,
\end{gathered}    
\end{equation}
such that the
following holds:
\begin{enumerate}
    \item All cardinals are regular, with the possible exception of $\lambda_\infty$,
    \item $\lambda_\infty=\lambda_\infty^{\aleph_0}$.
    \item\label{item:blabla}
    GCH, plus 
    $\theta_{n}$ is the successor of a regular cardinal for
    $n=6,4,2,0,\infty$.  %
    
    I.e., the assumptions for
    Theorem~\ref{thm:left} are satisfied if we set
    \begin{equation}\label{eq:mutotheta}
        \mu_i:=\theta_{8-2i}\text{ for $i=1,2,3,4$, and }\mu_\infty:=\theta_\infty.
    \end{equation}
\end{enumerate}

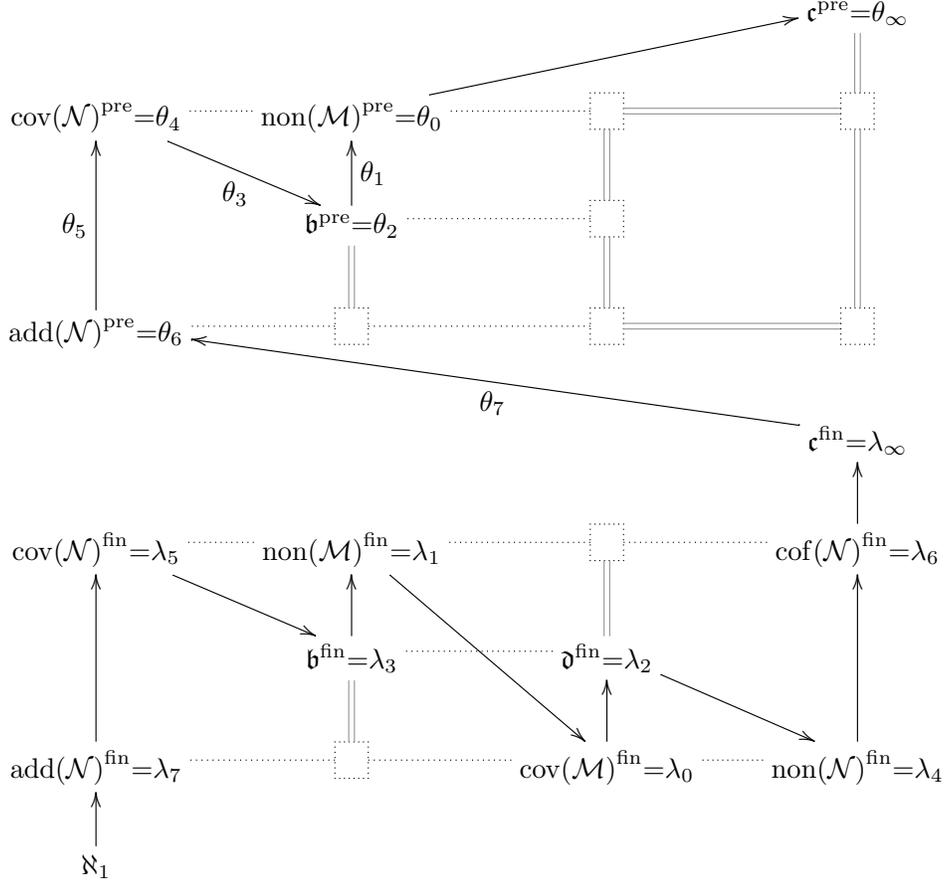
\begin{figure}
\resizebox{\textwidth}{!}{$
\xymatrix@=5ex{
&
&
&\txt{$\cfraki{=}\theta_\infty$}
\\
\txt{$\covNi{=}\theta_4$}
\ar@{.}[r]\ar[dr]_-{\textstyle\theta_3}
&\txt{$\nonMi{=}\theta_0 $}
\ar@{.}[r]\ar[rru]
&\mye
\ar@{=}@[Gray][r]
&\mye
\ar@{=}@[Gray][u]
\\
&\txt{$\bfraki{=}\theta_2$}
\ar@{.}[r]\ar[u]_-{\textstyle\theta_1}
&\mye
\ar@{=}@[Gray][u]
\\
\txt{$\addNi{=}\theta_6$}
\ar@{.}[r]\ar[uu]^-{\textstyle\theta_5}
&\mye
\ar@{.}[r]\ar@{=}@[Gray][u]
&\mye
\ar@{=}@[Gray][r]\ar@{=}@[Gray][u]
&\mye
\ar@{=}@[Gray][uu]
\\
&
&
&\txt{$\cfrakf{=}\lambda_\infty$}
\ar[ulll]^-{\textstyle\theta_7}
\\
\txt{$\covNf{=}\lambda_5$}
\ar@{.}[r]\ar[rd]
&\txt{$\nonMf{=}\lambda_1$}
\ar@{.}[r]\ar[ddr]
&\mye
\ar@{.}[r]
&\txt{$\cofNf{=}\lambda_6$}
\ar[u]
\\
&\txt{$\bfrakf{=}\lambda_3$}
\ar@{.}[r]\ar[u]
&\txt{$\dfrakf{=}\lambda_2$}
\ar@{=}@[Gray][u]\ar[dr]
\\
\txt{$\addNf{=}\lambda_7$}
\ar@{.}[r]\ar[uu]
&\mye\ar@{.}[r]
\ar@{=}@[Gray][u]
&\txt{$\covMf{=}\lambda_0$}
\ar@{.}[r]\ar[u]
&\txt{$\nonNf{=}\lambda_4$}
\ar[uu]
\\
\aleph_1
\ar[u]
}
$}
\caption{\label{fig:setup2}Our setup. The cardinals $\lambda_n$ and $\theta_n$ are increasing along the arrows (strictly increasing above $\lambda_\infty$).
The preparatory forcing $\Ppre$ forces 
$\mathfrak x=\innitialmark{\mathfrak x}$ for each left hand side characteristic $\mathfrak x$ (and forces the whole right side to be $\theta_\infty$); while
the final forcing $\Pfin$ forces 
$\mathfrak x=\finalmark{\mathfrak x}$ for every
characteristic $\mathfrak x$ (on either side).
I.e., the upper Cicho\'n's diagram shows the situation forced by $\Ppre$,
and the lower diagram shows the one forced by $\Pfin$.}
\end{figure}

So we can apply Theorem~\ref{thm:left}, resulting in
the forcing
$\Ppre$.
(Thus $\Ppre$ forces the situation shown in the upper Cicho\'n diagram
of Figure~\ref{fig:setup2}.)

We will now construct a forcing $\Pfin=\Ppre\cap N^*$ (a complete subforcing
of $\Ppre$) which forces
$(\bfrak_i,\dfrak_i)=(\lambda_{8-2i+1},\lambda_{8-2i})$ for all $i=1,\dots,4$, and $\cfrak=\lambda_\infty$ (i.e., the situation shown in the lower Cicho\'n diagram
of Figure~\ref{fig:setup2}).

We fix $N_{n,\alpha}$ for $0\le n\le 7,\ \alpha\in\lambda_n$,
as well as $N_8$,
satisfying the following for any $n\le 7$:
\begin{itemize}
    \item Each $N_{n,\alpha}$ as well as $N_8$ is an elementary submodel of $H(\chi)$ and contains (as elements) the sequences
    of $\theta$'s and $\lambda$'s, as well as $\Ppre$ and $S_i$ (the directed orders provided by Theorem~\ref{thm:left}) for $i=1,2,3,4$.
    \item $N_{n,\alpha}$ contains
        $(N_{m,\beta})_{m<n,\, \beta\in\lambda_m}$ as well as
        $(N_{n,\beta})_{\beta<\alpha}$.
        
        $N_8$ contains  
        $(N_{m,\beta})_{m\le 7,\, \beta\in\lambda_m}$.
    \item $|N_{n,\alpha}|=\theta_n$, and $N_{n,\alpha}$ is
        ${<}\theta_n$-closed (thus $\theta_n\subseteq N_{n,\alpha}$).\footnote{For $n\leq6$, ${<}\theta_{n+1}^+$-closed is enough; for $n=7$, ${<}\lambda_7$-closed is sufficient.} 
    \item We set $N_{n}:=\bigcup_{\alpha\in \lambda_n} N_{n,\alpha}$.    Note that $N_{n}$ is ${<}\lambda_n$-closed
       and has size $\theta_n$.
    \item $N_8$ is ${<}\aleph_1$-closed and has
        size $\lambda_\infty$.
    \item
        We set $N^*:=N_0\cap\dots\cap N_7\cap N_8$.
    \item For $0\le m\le 8$, we set
    $P_m:= \Ppre\cap N_{0}\cap\dots\cap N_{m}$ and
    $\Pfin:=P_{8}=\Ppre\cap N^*$.
    
    Note that $N_0\cap\dots\cap N_m$ is again an
    elementary submodel of $H(\chi)$,\footnote{If $M,N\preccurlyeq H(\chi)$ and $M\in N$ then $M\cap N\preccurlyeq M$ and $M\cap N\preccurlyeq N$.} and accordingly each
    $P_m$ is a complete subforcing of $\Ppre$.
\end{itemize}


\textbf{Regarding $\LCU$:}
We fix $i\in \{1,2,3,4\}$ (the case $i=3$, as an example, is described more explicitly below).
Let us call the set of regular cardinals $\kappa$
    satisfying $\LCU_i(P,\kappa)$ the ``$\LCU_i$-spectrum of $P$'', and let
    $\innitialmark{X}_i$ be the $\LCU_i$-spectrum of $\Ppre$. So
    \[
    \{ \theta_\infty,\theta_0,\dots,\theta_{8-2i} \} \subseteq \innitialmark{X}_i
    \]

\begin{itemize}
\item
    In the first step $n=0$, let us consider the $\LCU_i$-spectrum $X^0_i$
    of $P_0$: As $\theta_\infty\in \innitialmark{X}_i$, we get $\lambda_0\in X^0_i$,
    and as $\theta_0,\dots,\theta_{8-2i}$ are in $\innitialmark{X}_i$, they are in $X^0_i$
    as well (both according to Lemma~\ref{lem:trivial}(\ref{item:eins}), using
    $\kappa=\lambda_i$).
\item
    For the next step $n=1$, we similarly get
    that the $\LCU_i$-spectrum $X^1_i$ of $P_1$ contains
    $\lambda_0,\lambda_1$, and, if $i\ne 4$,
    also $\theta_1,\dots, \theta_{8-2i}$.
\item
    In this way we get that the final $\LCU_i$-spectrum $\finalmark{X}_i$
    of $\Pfin$ contains
    $\lambda_0,\dots,\lambda_{8-2i+1}$.
\item
    This implies (by Corollary~\ref{cor:alsotrivial})
    that $\Pfin$ forces
    \[
    \bfrak_i\le \min(\lambda_0,\dots,\lambda_{8-2i+1})=\lambda_{8-2i+1}
    \text{ and }
    \dfrak_i\ge \max(\lambda_0,\dots,\lambda_{8-2i+1}) = \lambda_{8-2i}.
    \]
    So we get half of the desired inequalities.
\end{itemize}
    This may be more transparent if we consider an explicit example, say $i=3$.  In each line of the following table, 
    each cardinal in the right column is guaranteed to be an element of the $\LCU_3$ spectrum of
    the forcing notion in the left column:
    \[
    \begin{array}{cclrcccc}
         \Ppre&&&&
        \theta_\infty,&\theta_0,&\theta_1,&\theta_2  \\
        P_0&=&\Ppre \cap N_0&&
        \lambda_0, &\theta_0, &\theta_1,&\theta_2 \\
        P_1&=&\Ppre \cap N_0\cap N_1 &&
        \lambda_0,& \lambda_1,& \theta_1,&\theta_2 \\
        P_2&=&\Ppre \cap N_0\cap N_1\cap N_2 &&
        \lambda_0,& \lambda_1,& \lambda_2,&\theta_2 \\
        P_3&=&\Ppre \cap N_0\cap N_1\cap N_2 &&
        \lambda_0,& \lambda_1,& \lambda_2,&\lambda_3 \\
        \vdots&&&
        \multicolumn{4}{c}{\vdots}\\
        \Pfin &&&&
        \lambda_0,& \lambda_1,& \lambda_2,&\lambda_3 \\
    \end{array}
    \]
    Since $\lambda_3$ is the smallest of these 4 cardinals, and $\lambda_2$ the largest, we get 
    that $\Pfin$ forces
    $
    \bfrak_i\le \lambda_3$ and 
    $\dfrak_i\ge  \lambda_2$. 


\textbf{Regarding $\COB$:}
Again we fix $i\in \{1,2,3,4\}$. Let  $m:=8-2i$. 
In particular, $0\le m\le 6$, $m$ is even,
so according to~\eqref{eq:lambalambalamba} we have $\lambda_{m+1}\leq\lambda_m$ and 
\begin{equation}\label{eq:blabla}
  \lambda_m=\max_{0\le n\le m}(\lambda_n)=\max_{0\le n\le m+1}(\lambda_n).
\end{equation}

Recall that $\COB_i(\Ppre, S_i)$ holds where $\comp(S_i)=\theta_{m}$ and $\cof(S_i)=\theta_\infty$ (cf.~Theorem~\ref{thm:left} and~\eqref{eq:mutotheta}).

We claim that
    \[T:=S_i\cap N_0\cap \dots\cap N_{m+1}\] satisfies
    \[
    \comp(T)\ge \min_{0\le n\le  m+1}(\lambda_n)=\lambda_{m+1}
    \quad\text{and}\quad
    \cf(T)\le \max_{0\le n\le  m+1}(\lambda_n)=\lambda_m.
    \]

    Completeness is clear by applying Lemma~\ref{lem:trivial}(\ref{item:zwei})
    iteratively: $\comp(S_i)>\lambda_0$, so $\comp(S_i\cap N_0)\ge \lambda_0$. Then $\comp(S_i\cap N_0\cap N_1)\ge\min\{\lambda_0,\lambda_1\}$, and so on.
    
Regarding the cofinality: 
\begin{itemize}
    \item 
        Let $\Lambda$ be the product $\prod_{n=0}^{m}\lambda_n$. 
        So $|\Lambda|=\lambda_m$ by~\eqref{eq:blabla}.

        For $\eta\in\Lambda$, set 
        $
            N^\eta:=\bigcap_{n=0}^m N_{n,\eta(n)}
        $. 
        Note that
        $
        N_{0}\cap\dots\cap N_{m}=\bigcup_{\eta\in\Lambda} N^\eta
        $, 
        and that $\Lambda$ is an element, and 
        thus a subset,
        of each elementary submodel.\footnote{Element is clear, as all
        $N$'s contain the sequence of $\lambda$'s.
        Subset follows from the fact that each $N$ 
        contains $\lambda_\infty$ and thus $\lambda_m$ as a subset, and 
        that $|\Lambda|=\lambda_m$.}
    
    \item For $\eta\in \Lambda$, set $T_\eta:= S_i\cap N^\eta$. Since $N_\eta$ is ${<}\theta_m$-closed and $\comp(S_i)\geq\theta_m$, we get $\comp(T_\eta)\geq\theta_m>\theta_{m+1}$. Hence, by Lemma~\ref{lem:trivial}(\ref{item:zero}) applied to $N=N_{m+1}$, $S=T_\eta$, $\kappa=\lambda=\lambda_{m+1}$ and $\theta=\theta_{m+1}$, we conclude $\cof(T_\eta\cap N_{m+1})=\lambda_{m+1}$. 
    
    Choose $C_\eta\subseteq T_\eta$ cofinal in $T_\eta\cap N_{m+1}$ of size $\lambda_{m+1}$. Hence, $C:=\bigcup_{\eta\in\Lambda}C_\eta$ is cofinal in $T$ because $T=\bigcup_{\eta\in\Lambda}T_\eta\cap N_{m+1}$, so $\cof(T)\leq|C|\leq|\Lambda|\cdot\lambda_{m+1}=
    \lambda_m\cdot \lambda_{m+1}=\lambda_m$
    by~\eqref{eq:blabla}.



\end{itemize}

    Now we show, by induction on $n\geq m+1$, that $S_i\cap N_0\cap\cdots\cap N_n$ has completeness ${\geq}\lambda_{m+1}$ and cofinality ${\leq}\lambda_m$. The step $n=m+1$ was done above; for the steps $n> m+1$, by induction we know that
    $S':=S_i\cap N_0\cap\dots\cap N_{n-1}$ has cofinality
    at most $\lambda_m$
    and completeness at least $\lambda_{m+1}$.
    So by Lemma~\ref{lem:trivial}(\ref{item:drei}),
    the same holds for $S'\cap N_{n}$.

    To summarize: For any $i=1,\dots,4$, the cofinality of
    $S_i\cap N^*$ is at most
    $\lambda_{8-2i}$, and the completeness at least
    $\lambda_{8-2i+1}$.
    By Lemmas~\ref{lem:trivialalso}(\ref{item:lcu}) and~\ref{cor:alsotrivial}(\ref{item:lcu}) we get
    \[\Pfin\Vdash \bfrak_i\ge\lambda_{8-2i+1}\ \& \ \dfrak_i\le\lambda_{8-2i}.\]
    So we get the remaining inequalities we need.


\textbf{Regarding the continuum:}
There is a sequence $(x_\xi)_{\xi<\theta_\infty}$ of (nice) $\Ppre$-names of reals that are forced to be pairwise different due to absoluteness~\eqref{eq:absolute}.
Note that this sequence belongs to $N^*$, so $(x_\xi)_{\xi\in\theta_\infty\cap N^*}$ is a sequence of $\Pfin$-names of reals that are forced (by $\Pfin$) to be pairwise different. Hence, $\Pfin$ forces $\cfrak\geq|\theta_\infty\cap N^*|=\lambda_\infty$.\footnote{This argument can be written in terms of the $\LCU$ property for the identity relation on $\omega^\omega$:
As $\LCU_{\mathrm{Id}}(\Ppre,\kappa)$ holds for all regular $\kappa\leq\lambda_\infty$ (even up to $\theta_\infty$), we get $\LCU_{\mathrm{Id}}(\Pfin,\kappa)$ for all these cardinals,
which implies
$\lambda_\infty\leq\cfrak$.} The converse inequality also holds because $|\Pfin|=\lambda_\infty=\lambda_\infty^{\aleph_0}$.%
\end{proof}

\begin{remark}\label{remGCH}
   If we base the left-hand forcing $\Ppre$
   on~\cite{diegoetal} (see Remark~\ref{rem:BCM}), 
   then our proof (when we change item~(\ref{item:blabla}) on p.~\pageref{item:blabla} to the assumptions
   listed in Remark~\ref{rem:BCM})
   shows 
   that GCH can be weakened to the following:
   There are at least 9 cardinals $\theta>\mu_9$ satisfying $\theta^{<\theta}=\theta$. Or, to be even be more pedantic: 
   There are regular cardinals $\theta_7<\ldots<\theta_0<\theta_\infty$ larger than $\mu_9$ such that $\theta_7^{<\mu_1}=\theta_7$, $\theta_\infty^{<\theta_2}=\theta_\infty$ and $\theta_i^{\theta_{i+1}}=\theta_i$ for $i\neq7,\infty$. 
\end{remark}

\section{Variants}\label{sec:more}

\subsection{Another order}\label{ss:other}
The paper~\cite{KeShTa:1131}  constructs (assuming GCH)  a ccc forcing notion $P$ which forces
another ordering of the left hand side.
More concretely, $P$ is ccc and it has $\LCU$
and $\COB$ witnesses for the following:\footnote{In~\eqref{eq:ordereofchars}, the order/numbering of $(\bfrak,\dfrak)$
and $(\covN,\nonN)$ is swapped; for this new ordering we again get Theorem~\ref{thm:left}. We use the same $R^\LCU$- and $R^\COB$-relations
 as in~\cite{GKS}, except for the $R^\LCU$-relation for the pair $(\covN,\nonN)$: Now we have to 
use a relation which is an $\omega_1$-union of Borel relations
(which was  originally defined in~\cite{KO} and fit into a formal
preservation framework in~\cite{CM19}; 
see details in~\cite[Def.~2.3]{KeShTa:1131}). 
This is the only place in this paper where we have to use 
a non-Borel relation $R$; but this is no problem as $R$ is sufficiently absolute
in the sense described after~\eqref{eq:absolute}.}
\[\addN<\bfrak<\covN<\nonM<\covM=\cfrak.\]

If we use this forcing $P$ instead of $\Ppre$, then the same
argument shows that we can
find a complete subforcing $\Pfin$ that extends the order to the right hand side:

\begin{theorem}\label{anothermain}
   Assume GCH and let $(\mu_n)_{1\le n\le 9}$ be a weakly  increasing sequence
of cardinals with $\mu_n$ regular for $n\leq8$ and $\mu_9^{\aleph_0}=\mu_9$. Then there is a ccc poset $\Pfin$
forcing that
\begin{gather*}
\aleph_1\leq\addN=\mu_1\leq\bfrak=\mu_2\leq\covN=\mu_3\leq\nonM=\mu_4\leq \\
\leq\covM=\mu_5\leq\nonN=\mu_6\leq\dfrak=\mu_7\leq\cofN=\mu_8\leq\cfrak=\mu_9.
\end{gather*}
\end{theorem}

\begin{remark}
As in Remark~\ref{remGCH}, full GCH is not needed, but it is enough that there are $9$ regular cardinals larger than $\mu_9$ satisfying some arithmetical properties. However, it is not enough that $\theta^{<\theta}=\theta$ for these $9$ cardinals, but it is required in addition that one of them is $\aleph_1$-inaccessible.\footnote{Recall that a cardinal $\theta$ is \emph{$\kappa$-inaccessible} if $\mu^\nu<\theta$ for every $\mu<\theta$ and $\nu<\kappa$.} For details, refer to~\cite{modKST,1166}.
\end{remark}

\subsection{A weaker notion than \texorpdfstring{$\COB$}{COB} sufficient for the proof}

Several papers about constellations of Cicho\'n's diagram preceding~\cite{GKS,diegoetal}, 
such as~\cite{Br,MR3047455,MR3513558},
have considered similar, but simpler, forcing constructions. While $\LCU$ witnesses are added in the same way, these do not provide for $\COB$. Instead, a weaker property, which we call $\DOM$ below, is implicit in these constructions. We now show that this notion is sufficient to carry out the proof of the main result.

\begin{definition}
   Let $R$ be a relation on $\omega^\omega$ and let $\kappa$ be a cardinal.
   \begin{enumerate}
       \item A set $A\subseteq\omega^\omega$ is \emph{$\kappa$-$R$-dominating} if, whenever $F\subseteq\omega^\omega$ has size ${<}\kappa$, there is some real $a\in A$ dominating over $F$, that is, $(\forall x\in F)xRa$. Dually, we say that $A$ is \emph{$\kappa$-$R$-unbounded} if it is $\kappa$-$R^\perp$-dominating.
       \item Assume that $R$ is sufficiently absolutely defined and let $P$ be a forcing notion. We define $\DOM_R(P,\kappa,S)$ to mean the following: There is a sequence $(f_\alpha)_{\alpha\in S}$ of $P$-names of reals such that, whenever $\gamma<\kappa$ and $(x_\xi)_{\xi<\gamma}$ is a sequence of $P$-names of reals, there is some $\alpha\in S$ such that $P\Vdash(\forall\xi<\gamma)x_\xi R f_\alpha$.
   \end{enumerate}
\end{definition}
(Note that $\DOM_R(P,\kappa,S)$ is stronger than just saying ``$P$ adds a $\kappa$-$R$-dominating family''.)

The following is straightforward:
\begin{itemize}
    \item $\COB_R(P,S)$ implies $\DOM_R(P,\comp(S),\cf(S))$.
    \item If $\kappa$ is regular then $\LCU_R(P,\kappa)$ implies $\DOM_{R^\perp}(P,\kappa,\kappa)$.
    \item $\DOM_R(P,\kappa,S)$ implies $P\Vdash\bigl(\,  \kappa\leq\bfrak_R\ \&\ \dfrak_R\leq|S|\, \bigr)$.
    
    (This generalizes Lemma~\ref{lem:blubb}.)
\end{itemize}

For this weaker notion we have the following result similar to Lemma~\ref{lem:trivial}.

\begin{lemma}\label{lem:dom}
      With the same hypothesis as in Lemma~\ref{lem:trivial}, assuming also $\nu\in N$:
      \begin{enumerate}
          \item $\DOM_R(P,\nu,S)$ implies $\DOM_R(P\cap N,\min\{\kappa,\nu\},S\cap N)$.
          \item If $|S|\leq\theta$ then $\DOM_R(P,\nu,S)$ implies $\DOM_R(P\cap N,\nu,S\cap N)$.
          \item If $\nu>\theta$ then $\DOM_R(P,\nu,S)$ implies $\COB_R(P\cap N,\lambda)$.
          
          In particular,
          if $S$ is directed and $\comp(S)>\theta$ then $\COB_R(P,S)$ implies $\COB_R(P\cap N,\lambda)$.
      \end{enumerate}
\end{lemma}
\begin{proof}
    In the following, assume that $(f_\alpha)_{\alpha\in S}$ witnesses $\DOM_R(P,\nu,S)$.
    
     (1) If $(x_\xi)_{\xi<\gamma}$ is a sequence of $P$-names of reals and $\gamma<\min\{\kappa,\nu\}$ then the sequence is in $N$, so there is some $\alpha\in S\cap N$ such that $P\Vdash x_\xi R f_\alpha $ for all $\xi<\gamma$. By absoluteness, $P\cap N$ forces the same.
     
     (2) is clear because $S\subseteq N$ (as $|S|\leq\theta$ and $S\in N$).
     
     (3) Fix $i<\lambda$. Since $|N_i|\leq\theta<\nu$, there is some $\alpha_i\in S$ such that $P\Vdash xRf_{\alpha_i}$ for all $x\in N_i$ that are $P$-names for reals. In fact, we can find such $\alpha_i$ in $S\cap N_{i+1}$. Hence, $(f_{\alpha_i})_{i<\lambda}$ witnesses $\COB_R(P\cap N,\lambda)$.
\end{proof}

As in~\cite{modKST}, a simpler version of $\Ppre$ can be constructed in such a way that 
\begin{itemize}
    \item[(a)] of Theorem~\ref{thm:left} holds, and
    \item[(b')] For $i=1,2,3,4$ there is some set $S_i$ of size $\mu_\infty$ such that $\DOM_{R^\COB_i}(P,\mu_i,S_i)$ holds.
\end{itemize}
Thanks to Lemma~\ref{lem:dom} (in particular item (3)), the same proof of Theorem~\ref{main} can be carried out in this simpler context.


\section{Open questions}\label{sec:q}

\cite[Sect.~3]{GKS} asks the following questions: Can you show the consistency of
 Cicho\'n's maximum \dots
\begin{enumerate}
    \item[\checkmark(a)] \dots without using large cardinals?
    \item[\checkmark(b)] \dots for specific (regular) values, such as  $\mu_i=\aleph_{i+1}$?
    \item[$\thicksim$(c)] \dots for other orderings of the ten entries?
    \item[$\thicksim$(d)] \dots together with further distinct values of additional (''classical'') cardinal characteristics?
\end{enumerate}

This work, more concretely Theorem~\ref{main}, solves questions (a) and (b).

A first result for (d), namely adding $\aleph_1<\mfrak<\pfrak<\hfrak<\addN$,
is done 
in~\cite{1166} (which also gives a more complicated construction to achieve (b)).

\medskip

Of course, it would be interesting to add more characteristics. For example, we can ask:
\begin{question}
    Can we add
    the splitting number $\mathfrak{s}$ and the reaping number $\mathfrak{r}$?
\end{question}
The pair $(\mathfrak s,\mathfrak r)$ might be most promising among the classical characteristics, as it is of the form $(\bfrak_R, \dfrak_R)$ for  a Borel
relation $R$ which is well understood.

\medskip

Question (c) remains largely open.
There are four possible configurations where $\nonM<\covM$, 
and at the moment only $2$ are known to be consistent (see Theorems~\ref{main} and~\ref{anothermain}).


\begin{question}
    Are the following two constellations consistent?
    \[\aleph_1<\addN<\covN<\bfrak<\nonM<\covM<\nonN<\dfrak<\cofN<\cfrak,\]
    \[\aleph_1<\addN<\bfrak<\covN<\nonM<\covM<\dfrak<\nonN<\cofN<\cfrak.\]
\end{question}



It is not clear whether our method in Section~\ref{sec:new} can be applied to solve this question (the same applies to Boolean ultrapowers), since we start with a poset forcing an order of the left side of Cicho\'n's diagram and our method only manages to \emph{dualize} this order to the right side (e.g., if on the left we force $\covN<\bfrak$, then on the right we can only expect to force the dual inequality $\dfrak<\nonN$). 

The case when $\covM<\nonM$ seems to be more complex.\footnote{Recall that finite support iterations add Cohen reals at limit steps, so they force $\nonM\leq\covM$ (when the length has uncountable cofinality).} We
do not even know how to force the consistency of $\aleph_1<\covM<\nonM$.
J.~Brendle however does, see~\cite{BrShat} 
for slides of a presentation of his method of ``shattered iterations''. 
Brute force counting shows that there are $57$ configurations of ten different values in Cicho\'n's diagram (satisfying the obvious inequalities) where $\covM<\nonM$, but none of them have been proved to be consistent so far. 
\begin{question}
     Is any constellation
    of Cicho\'n's maximum consistent where 
    $\covM<\nonM$?
\end{question}

\bibliography{morebib}

\providecommand{\bysame}{\leavevmode\hbox to3em{\hrulefill}\thinspace}
\providecommand{\MR}{\relax\ifhmode\unskip\space\fi MR }
\providecommand{\MRhref}[2]{%
  \href{http://www.ams.org/mathscinet-getitem?mr=#1}{#2}
}
\providecommand{\href}[2]{#2}
\begin{thebibliography}{FFMM18}

\bibitem[Bar84]{MR719666}
Tomek Bartoszy\'nski, \emph{Additivity of measure implies additivity of
  category}, Trans. Amer. Math. Soc. \textbf{281} (1984), no.~1, 209--213.
  \MR{719666}

\bibitem[BCM18]{diegoetal}
J\"{o}rg Brendle, Miguel~A. Cardona, and Diego~A. Mej\'{i}a,
  \emph{Filter-linkedness and its effect on preservation of cardinal
  characteristics}, preprint,
  \href{https://arxiv.org/abs/1809.05004}{arXiv:1809.05004}, 2018.

\bibitem[BJ95]{BJ}
Tomek Bartoszy\'nski and Haim Judah, \emph{Set theory}, A K Peters, Ltd.,
  Wellesley, MA, 1995, On the structure of the real line. \MR{1350295}

\bibitem[BJS93]{MR1233917}
Tomek Bartoszy{\'n}ski, Haim Judah, and Saharon Shelah, \emph{The {C}icho\'n
  diagram}, J. Symbolic Logic \textbf{58} (1993), no.~2, 401--423. \MR{1233917
  (94m:03077)}

\bibitem[Bre91]{Br}
J{\"o}rg Brendle, \emph{Larger cardinals in {C}icho\'n's diagram}, J. Symbolic
  Logic \textbf{56} (1991), no.~3, 795--810. \MR{1129144 (92i:03055)}

\bibitem[{Bre}19]{BrShat}
{Brendle, J\"{o}rg}, \emph{{Forcing and cardinal invariants, parts 1 and 2}},
  tutorial at Advanced Class - Young Set Theory Workshop 2019, July
  28\textsuperscript{th} and 29\textsuperscript{th} 2019, Slides:
  \url{https://sites.google.com/view/estc2019/advanced-class-yst/program}.

\bibitem[CKP85]{MR781072}
J.~Cicho{\'n}, A.~Kamburelis, and J.~Pawlikowski, \emph{On dense subsets of the
  measure algebra}, Proc. Amer. Math. Soc. \textbf{94} (1985), no.~1, 142--146.
  \MR{781072 (86j:04001)}

\bibitem[CM19]{CM19}
Miguel~A. Cardona and Diego~A. Mej\'{i}a, \emph{On cardinal characteristics of
  {Y}orioka ideals}, Mathematical Logic Quarterly \textbf{65} (2019), no.~2,
  170--199.

\bibitem[Coh63]{MR0157890}
Paul Cohen, \emph{The independence of the continuum hypothesis}, Proc. Nat.
  Acad. Sci. U.S.A. \textbf{50} (1963), 1143--1148. \MR{0157890}

\bibitem[FFMM18]{MR3796283}
Vera Fischer, Sy~D. Friedman, Diego~A. Mej\'{\i}a, and Diana~C. Montoya,
  \emph{Coherent systems of finite support iterations}, J. Symb. Log.
  \textbf{83} (2018), no.~1, 208--236. \MR{3796283}

\bibitem[FGKS17]{five}
Arthur Fischer, Martin Goldstern, Jakob Kellner, and Saharon Shelah,
  \emph{Creature forcing and five cardinal characteristics in {C}icho\'{n}'s
  diagram}, Arch. Math. Logic \textbf{56} (2017), no.~7-8, 1045--1103.
  \MR{3696076}

\bibitem[Git19]{moti2}
Moti Gitik, \emph{{A remark on large cardinals in ``Cichon Maximum''}},
  unpublished note, April 18\textsuperscript{th} 2019,
  \url{http://www.math.tau.ac.il/~gitik/sstr.pdf}.

\bibitem[GKMS19]{1166}
Martin Goldstern, Jakob Kellner, Diego Mej\'{\i}a, and Saharon Shelah,
  \emph{Controlling cardinal characteristics without adding reals}, preprint,
  \href{https://arxiv.org/abs/1904.02617v1}{arXiv:1904.02617v1}, 2019.

\bibitem[GKS19]{GKS}
Martin Goldstern, Jakob Kellner, and Saharon Shelah, \emph{Cicho\'{n}'s
  maximum}, Ann. of Math. \textbf{190} (2019), no.~1, 113--143,
  \href{https://arxiv.org/abs/1708.03691}{arXiv:1708.03691}.

\bibitem[GMS16]{MR3513558}
Martin Goldstern, Diego~Alejandro Mej{\'\i}a, and Saharon Shelah, \emph{The
  left side of {C}icho\'n's diagram}, Proc. Amer. Math. Soc. \textbf{144}
  (2016), no.~9, 4025--4042. \MR{3513558}

\bibitem[G{\"o}d40]{goedel}
Kurt G{\"o}del, \emph{The {C}onsistency of the {C}ontinuum {H}ypothesis},
  Annals of Mathematics Studies, no. 3, Princeton University Press, Princeton,
  N. J., 1940. \MR{0002514 (2,66c)}

\bibitem[JS90]{MR1071305}
Haim Judah and Saharon Shelah, \emph{The {K}unen-{M}iller chart ({L}ebesgue
  measure, the {B}aire property, {L}aver reals and preservation theorems for
  forcing)}, J. Symbolic Logic \textbf{55} (1990), no.~3, 909--927.
  \MR{1071305}

\bibitem[Kam89]{MR1022984}
Anastasis Kamburelis, \emph{Iterations of {B}oolean algebras with measure},
  Arch. Math. Logic \textbf{29} (1989), no.~1, 21--28. \MR{1022984}

\bibitem[KO14]{KO}
Shizuo Kamo and Noboru Osuga, \emph{Many different covering numbers of
  {Y}orioka's ideals}, Arch. Math. Logic \textbf{53} (2014), no.~1-2, 43--56.
  \MR{3151397}

\bibitem[KST19]{KeShTa:1131}
Jakob Kellner, Saharon Shelah, and Anda T{\u{a}}nasie, \emph{Another ordering
  of the ten cardinal characteristics in cicho\'{n}'s diagram}, Comment. Math.
  Univ. Carolin. \textbf{60} (2019), no.~1, 61--95.

\bibitem[KTT18]{KTT}
Jakob Kellner, Anda~Ramona T\u{a}nasie, and Fabio~Elio Tonti, \emph{Compact
  cardinals and eight values in {C}icho\'{n}'s diagram}, J. Symb. Log.
  \textbf{83} (2018), no.~2, 790--803. \MR{3835089}

\bibitem[Kun11]{Kunen}
Kenneth Kunen, \emph{Set theory}, Studies in Logic (London), vol.~34, College
  Publications, London, 2011. \MR{2905394}

\bibitem[Mej13]{MR3047455}
Diego~Alejandro Mej{\'\i}a, \emph{Matrix iterations and {C}ichon's diagram},
  Arch. Math. Logic \textbf{52} (2013), no.~3-4, 261--278. \MR{3047455}

\bibitem[Mej19a]{mejiavert}
Diego~A. Mej{\'{i}}a, \emph{Matrix iterations with vertical support
  restrictions}, Proceedings of the 14th and 15th Asian Logic Conferences
  (Byunghan Kim, J\"{o}rg Brendle, Gyesik Lee, Fenrong Liu, R~Ramanujam,
  Shashi~M Srivastava, Akito Tsuboi, and Liang Yu, eds.), World Sci. Publ.,
  2019, pp.~213--248.

\bibitem[Mej19b]{modKST}
Diego~A. Mej\'{i}a, \emph{A note on ``{Another ordering of the ten cardinal
  characteristics in Cicho\'n's Diagram}'' and further remarks}, Ky\={o}to
  Daigaku S\=urikaiseki Kenky\=usho K\=oky\=uroku \textbf{2141} (2019), 1--15,
  \href{https://arxiv.org/abs/1904.00165}{arXiv:1904.00165}.

\bibitem[Mil81]{MR613787}
Arnold~W. Miller, \emph{Some properties of measure and category}, Trans. Amer.
  Math. Soc. \textbf{266} (1981), no.~1, 93--114. \MR{613787 (84e:03058a)}

\bibitem[Mil84]{MR735576}
\bysame, \emph{Additivity of measure implies dominating reals}, Proc. Amer.
  Math. Soc. \textbf{91} (1984), no.~1, 111--117. \MR{735576 (85k:03032)}

\bibitem[RS83]{MR697963}
Jean Raisonnier and Jacques Stern, \emph{Mesurabilit\'e et propri\'et\'e de
  {B}aire}, C. R. Acad. Sci. Paris S\'er. I Math. \textbf{296} (1983), no.~7,
  323--326. \MR{697963 (84g:03077)}

\bibitem[RS85]{MR800191}
\bysame, \emph{The strength of measurability hypotheses}, Israel J. Math.
  \textbf{50} (1985), no.~4, 337--349. \MR{800191}

\end{thebibliography}
\bibliographystyle{amsalpha}


\end{document}